\documentclass[11pt,reqno]{amsart}
\usepackage{graphicx,amsmath,amsfonts,latexsym,amssymb,amsthm,color}
\usepackage{latexsym}
\usepackage{verbatim}
\usepackage{enumerate}

\usepackage[a4paper]{geometry}
\definecolor{blau}{rgb}{0.05,0.2,0.7}
\definecolor{rot}{rgb}{0.99,0.02,0.02}
\definecolor{auchblau}{rgb}{0.03,0.3,0.7}
\usepackage[pdftex,linkbordercolor={0.8 0.5 0.2}]{hyperref} 
\hypersetup{
pdfauthor = {Joern Mueller, Alexander Strohmaier},
pdftitle = {Hahn-holomorphic functions},
pdfsubject ={Hahn-holomorphic functions}{Hahn series}{Fredholm Theorem}{Scattering Theory}{Spectral Theory},
colorlinks = true,
linkcolor = blau,
citecolor = auchblau
}

\newcommand{\Id}{\mathrm{Id}}

\renewcommand{\Re}{\operatorname{Re}}
\renewcommand{\Im}{\operatorname{Im}}
\DeclareMathOperator{\supp}{supp}

\newcommand{\Complex}{\mathbb{C}}
\newcommand{\Reell}{\mathbb{R}}
\newcommand{\Nat}{\mathbb{N}}
\newcommand{\Rational}{\mathbb{Q}}
\newcommand{\Ganz}{\mathbb{Z}}
\newcommand{\vol}{{\mathrm{vol}}}
\newcommand{\ZZ}{\mathcal{Z}}
\newcommand{\h}[1]{\mathfrak{#1}}
\newcommand{\eps}{\varepsilon}
\newcommand{\infn}[2]{\left\|#1\right\|_{#2}}
\newcommand{\rmi}{{\rm i}}

\newcommand{\sce}[2]{{\langle #1\,,\:#2\rangle}}

\newtheorem{theorem}{Theorem}[section]
\newtheorem{definition}[theorem]{Definition}
\newtheorem{example}[theorem]{Example}
\newtheorem{lemma}[theorem]{Lemma}
\newtheorem{corollary}[theorem]{Corollary}
\newtheorem{proposition}[theorem]{Proposition}
\newtheorem{rem}[theorem]{Remark}

\title[Hahn holomorphic functions]{The theory of Hahn meromorphic functions, a holomorphic Fredholm theorem and its applications}

\author[J.~M\"uller]{J\"orn M\"uller}
\address{Humboldt-Universit\"at zu Berlin,
Unter den Linden 6,
Institut f\"ur Mathematik,
D - 10099 Berlin, Germany} \email{jmueller@math.hu-berlin.de}

\author[A.~Strohmaier]{Alexander Strohmaier}
\address{Department of Mathematical Sciences,  Loughborough University,  Loughborough, Leicestershire, LE11 3TU,
UK} \email{a.strohmaier@lboro.ac.uk}

\thanks{Both authors were supported by the \emph{SFB 647: Space--Time--Matter. Analytic and Geometric Structures}}

\begin{document}

\begin{abstract}
We introduce a class of functions near zero on the logarithmic cover of the complex plane
that have convergent expansions into generalized power series. The construction covers cases where
non-integer powers of $z$ and also terms containing $\log z$ can appear.
We show that under natural assumptions some important theorems from complex analysis carry over to the class of these functions. In particular it
is possible to define a field of functions that generalize meromorphic functions and one can
formulate an analytic Fredholm theorem in this class. We show that this modified analytic Fredholm
theorem can be applied in spectral theory to prove convergent expansions of the resolvent for Bessel type operators
and Laplace-Beltrami operators for manifolds that are Euclidean at infinity. These results are important in scattering theory
as they are the key step to establish analyticity of the scattering matrix and the existence of generalized eigenfunctions at points
in the spectrum.
\end{abstract}

\maketitle


\section{Introduction}

Asymptotic expansions of the form
$$
f(z) \sim \sum_{k,m} a_{k,m} z^{\alpha_k} (-\log(z))^{\beta_m}, \quad \textrm{as} \quad z \to 0
$$
with non-integer $\alpha_k$ or $\beta_m$
for functions $f$ defined in some sector centered at $0$ in the complex plane
appear quite frequently in mathematics and mathematical physics.
Classical examples are solutions for differential equations (e.g. in Frobenius' method)
or expansions of algebraic functions at singularities. More recently it was shown
that low energy resolvent expansions in scattering problems are of this form
(see e.g.\ \cite{Kato-Jensen}, \cite{Jensen-Nenciu} for Schr\"odinger operators
in $\Reell^n$, \cite{Murata} for operators with constant leading coefficients
in  $\Reell^n$  and \cite{Guillarmou-Hassell} for the case of the Laplace operator
on a general manifold with a conical end). The resolvent expansion for $|\lambda|\to\infty$ of cone degenerate differential operators leads to similar asymptotics, see e.g.~\cite{gmk}.
In many of the examples above the expansions can be shown to be convergent under more restrictive assumptions
on the structure at infinity of the underlying geometry.

The algebraic theory of generalized power series is well developed and can be found
in the literature under the name Hahn series or Mal'cev-Neumann series (e.g.~\cite{Hahn-Orig}, \cite[Chapter 13]{passman}, \cite{ribenboim}). In this paper
we are concerned with the analytic theory of such generalized power series, namely
we will define a ring of functions, the Hahn holomorphic functions, that have convergent expansions into generalized power series, 
and we will show that
this ring is actually a division ring. We show that the quotient field, the field of Hahn-meromorphic
functions, has a nice description in terms of Hahn series and we generalize the notions
of Hahn-holomorphic and Hahn-meromorphic functions to the operator valued case.
The theory turns out to be very close to the case of analytic function theory. In particular
one of our main theorems states that an analog of the analytic Fredholm theorem holds in the class
of Hahn holomorphic functions.

The holomorphic Fredholm theorem plays an important role in geometric scattering theory as a tool to prove
the existence of a meromorphic continuation of resolvent kernels of elliptic differential operators such as the Laplace operator.
The extension is typically from the resolvent set across the continuous spectrum to a branched covering of the complex plane. 
As soon as such a meromorphic continuation
of the resolvent kernel is established, resonances can be defined as poles of its continuation, 
generalized eigenfunctions may be defined as meromorphic functions of a suitably chosen spectral parameter, 
and an analytic continuation of the scattering matrix  may be constructed.
This in many situations leads to a rich mathematical structure
that results in functional equations for the scattering matrix and Maass-Selberg relations for the generalized eigenfunctions
(see e.g. \cite{Mueller-LecN} for the case of manifolds with cusps of rank one, \cite{Melrose:APS,guil,MS} for manifolds with cylindrical ends, and
\cite{JMpaper} for manifolds with fibred cusps).
In particular the analytic continuation of  Eisenstein series may be regarded as a special case of this more general construction.

Often, as for example in the case of $\mathbb{R}^{2n+1}$, on asymptotically hyperbolic manifolds (\cite{Mazzeo-Melrose:Meromorphic,Guillarmou:Meromorphic}), geometrically finite hyperbolic manifolds 
(\cite{Guillarmou-Mazzeo:cusps}), and on globally symmetric spaces 
of odd rank (\cite{Mazzeo-Vasy:Symmetric,Strohmaier:Symmetric}), the branch points of the covering of the complex plane
are algebraic and can be resolved by a change of variables. In this way one can make sense of the statement that the resolvent
is meromorphic at the branch point.
In other examples, as in $\mathbb{R}^{2n}$, on symmetric spaces of even rank (\cite{Mazzeo-Vasy:Symmetric,Strohmaier:Symmetric}), and on manifolds with generalized cusps \cite{hrs}
the branch point is logarithmic  and this statement loses its meaning. 
The analytic Fredholm theorem can then only be applied away from the branching points. Our philosophy is that at such branching points it still makes sense to say when functions
are Hahn-holomorphic, i.e. have a convergent expansion into more general power series possibly containing $\log$-terms.
Our Hahn analytic Fredholm theorem therefore allows to analyze the resolvents at non-algebraic branching points.
Our theorem implies for example that the Hahn meromorphic properties of the resolvent
of the Laplace operator on a Riemannian manifold are stable under 
perturbations of the topology and the metric that are supported in compact regions.
The theory can be developed further to establish
Hahn-analyticity of the scattering matrix and of the generalized eigenfunctions in this context, but we decided to focus
on the theoretical properties of Hahn-meromorphic functions first and keep the presentation self-contained. The applications in scattering
theory will be developed elsewhere.

The article is organized as follows. Section \ref{abschnittHHol} deals with the definition and the general theory
of Hahn holomorphic functions and some of their basic properties. In section \ref{section:mero} we define Hahn
meromorphic functions and in section \ref{section:fred} we prove our generalization of the
meromorphic Fredholm theorem in the framework of Hahn holomorphic functions. Sections \ref{section:z}
and \ref{section:zlogz} deal with two important examples of convergent Hahn series: those that can be expanded into
real powers of $z$ and those that have such expansions with additional $\log(z)$-terms.
The theory has a nice application: convergent resolvent expansions for Bessel type operators and Laplace-Beltrami operators
on manifolds that are Euclidean at infinity can be shown to be simple consequences of the Hahn holomorphic Fredholm theorem.
These examples are treated in detail in section \ref{section:beispiel}; the main results here are Theorem \ref{propext1.2} and
Theorem \ref{hahnperturbed}.

We would like to thank the anonymous referee for suggestions leading
to considerable
simplifications in some of the arguments in section \ref{section:beispiel}.

\section{Hahn holomorphic functions}
\label{abschnittHHol}

Let $(\Gamma,+)$ be a linearly ordered abelian group and let $(G,\cdot)$ be a group.
Suppose $e:\Gamma\to G, \gamma\mapsto e_\gamma$ is a group homomorphism, in particular
\[
e_0=\mathbf{1}\in G,\qquad e_{\gamma_1+\gamma_2}=e_{\gamma_1}\cdot e_{\gamma_2} \quad\text{for all}\quad\gamma_1,\gamma_2\in\Gamma.
\]

The following definition and proposition are due to H.~Hahn (see \cite{Hahn-Orig})

\begin{definition}
Let $\mathcal{R}$ be a ring. A formal series
\[
\h{h}=\sum_{\gamma\in \Gamma} a_{\gamma} e_\gamma,\qquad  a_{\gamma}\in \mathcal{R}
\]
is called a \emph{Hahn-series}, if  the support of $\h{h}$,
\[
\supp(\h{h}):=\{g\in \Gamma\mid a_g\neq 0\in\mathcal{R}\},
\]
is a  well-ordered subset of $\Gamma$. The set of Hahn-series will be denoted by $\mathcal{R}[[e_\Gamma]]$.
\end{definition}

\begin{proposition}
The set of Hahn series $\mathcal{R}[[e_{\Gamma}]]$ is a ring with multiplication
\begin{align}
\Big(\sum_{\alpha\in \Gamma} a_{\alpha} e_\alpha\Big)
\Big(\sum_{\beta\in \Gamma} b_{\beta} e_\beta\Big)
&=\sum_{\gamma\in\Gamma}c_\gamma e_\gamma,
\qquad c_\gamma:=\!\!\!\sum_{\substack{(\alpha,\beta)\in \Gamma\times \Gamma \\ \alpha+\beta=\gamma}}
a_{\alpha}  b_{\beta}  \label{hahn-prod}\\
\intertext{and addition}
\sum_{\alpha\in \Gamma} a_{\alpha} e_\alpha
+\sum_{\beta\in \Gamma} b_{\beta} e_\beta
&=\sum_{\gamma\in\Gamma}
(a_{\gamma} +b_{\gamma})  e_\gamma\nonumber
\end{align}
If  $\mathcal{R}$ is a field, then $\mathcal{R}[[e_{\Gamma}]]$ is a field.
\end{proposition}

If the support of $\h{h}$ is contained in $\Gamma^+=\{\gamma \mid \gamma > 0\}$
then it is well known that $\mathbf{1}-\h{h}$ is invertible in $\mathcal{R}[[e_{\Gamma}]]$ and its inverse
is given by the Neumann series
\[
(\mathbf{1}-\h{h})^{-1} = \sum_{k=0}^\infty \h{h}^k.
\]
This is due to the fact that for any well-ordered subset $W$
of $\Gamma^+$ the semi-group generated by $W$ is also well-ordered, see e.g.\ \cite{passman}, Lemma 2.10.
Here convergence of a sequence $(\h{p}_n)\subset\mathcal{R}[[e_\Gamma]]$ to $\h{p}\in\mathcal{R}[[e_\Gamma]]$ is understood in the sense that
for every element $\alpha \in \Gamma$ there exists an $N>0$ such
that for all $n>N$ the coefficients of $e_\alpha$ in $\h{p}$ and $\h{p}_n$ are equal.

In the following let $\ZZ$ be the logarithmic covering surface of the complex plane without the origin.
We will use polar coordinates $(r,\varphi)$ as global coordinates to identify $\ZZ$ as a set with $\Reell_+ \times \Reell$.
Adding a single point $\{0\}$ to $\ZZ$ we obtain a set $\ZZ_0$ and a projection map $\pi: \ZZ_0 \to \Complex$ by extending the covering map $\ZZ \to \Complex\backslash \{0\}$ in sending $0 \in \ZZ_0$ to $0 \in \Complex$.
We endow $\ZZ$ with the covering topology and $\ZZ_0$ with the topology generated by the open sets in $\ZZ$
together with the open discs $D_\epsilon:=\{0\} \cup \{(r,\varphi)\mid 0\le r<\epsilon \}$.
This means a sequence $((r_n,\varphi_n))_n$ converges to zero if and only if $r_n \to 0$. The covering map is continuous
with respect to this topology.
For a point $z \in \ZZ_0$
we denote by $| z |$ its $r$-coordinate and by $\arg(z)$ its $\varphi$ coordinate. We will think of the positive
real axis as embedded in $\ZZ$ as the subset $\{ z \mid \arg(z)=0\}$. 
In the following $Y \subset \ZZ$ will always denote an open subset containing an open interval $(0,\delta)$
for some $\delta>0$ and such that $0 \notin Y$. The set $Y_0$ will denote $Y \cup \{0\}$.
In the applications we have in mind the set $Y$ is typically of the form
$D_\delta^{[\sigma]} \backslash \{0\}$ where
$D_\delta^{[\sigma]}=\{z \in \ZZ_0 \mid 0\le  |z| < \delta,\; |\varphi| < \sigma\}$. For the discussion and the general theorems
it is not necessary to restrict ourselves to this case.

In the remaining part of this article we assume that $G:=\big(\mathrm{Hol}(Y \cap D_\epsilon),\cdot\big)^{\times}$
is a set of non-vanishing holomorphic functions and that the
group homomorphism $e$ satisfies the condition
\[
\tag{E1}\label{Eone} \forall \gamma>0: \text{$e_\gamma$ is bounded on $Y$ and}\quad \lim_{z\to 0} |e_\gamma(z)|=0.
\]
\vspace{0.1ex}

\begin{definition}
Suppose that $\mathcal{R}$ is a vector space with norm $\|.\|$.
A Hahn series $\h{f}=\sum_{\alpha\in \Gamma} a_\alpha e_\alpha$  is called \emph{normally convergent} in $Y \cap D_\epsilon$ if its support is countable and
\begin{gather*}
\sum_{\alpha\in \Gamma} \|a_\alpha \| \infn{e_\alpha}{Y,\eps}< \infty,
\end{gather*}
where
$\infn{e_\alpha}{Y,\eps}:=\sup_{z\in Y \cap D_\epsilon}|e_\alpha(z)|.$
\end{definition}
Since a normally convergent series converges absolutely and uniformly, the value of the function
\[
f(z)= \sum_{\alpha\in \Gamma} a_\alpha e_\alpha(z),\qquad z\in Y\cap D_\eps
\]
does not depend on the order of summation and $f$ is holomorphic in $z\neq 0$.
\begin{definition}
Let  $S\subset \Gamma_0^+=\Gamma^+\cup\{0\}$ be a subset of the non-negative group elements.
\begin{itemize}
\item The family $
\{e_{\alpha}\}_{\alpha\in S}
$ is called \emph{weakly monotonous}, if there exists an $r_{S}>0$ such that for every $x\in(0,r_{S})$ there is a
\emph{radius} $\rho(x)$ with $0<\rho(x)\le x$ and with the property
\[
\alpha\in S \Rightarrow \infn{e_\alpha}{Y,\rho(x)} \le |e_\alpha(x)|.
\]
\item The set $S$ is called \emph{admissible for $e$} (or simply \emph{admissible}), if
$
\{e_{\alpha}\}_{\alpha\in S}
$ is weakly monotonous, and if for every $B\subset S$ also the family
\[
\{e_{\alpha-\min B}\}_{\substack{\alpha\in S\\ \alpha>\min B}}
\]
is weakly monotonous.
\end{itemize}

\end{definition}

\begin{definition}[Hahn holomorphic functions] \label{hhf}
Suppose that $\mathcal{R}$ is a Banach algebra.
A continuous function  $h: Y_0 \to \mathcal{R}$ which is holomorphic in $Y$, is called $(Y,\Gamma)$-Hahn holomorphic (or simply \emph{Hahn holomorphic}) if there is a Hahn-series \[
\h{h}=\sum_{\gamma\in \Gamma} a_{\gamma} e_\gamma,\qquad  a_{\gamma}\in \mathcal{R},
\]
with countable, admissible support,
which converges normally on $Y \cap D_\delta$ for some $\delta>0$, and
\[
h(z)=\sum_{\gamma\in \Gamma} a_{\gamma} e_\gamma(z), \quad z \in Y \cap D_\delta.
\]
\end{definition}

We will denote the Hahn series of a Hahn holomorphic function $h$ by the corresponding ``fraktur'' letter $\h{h}$.
Note that (E1) together with uniform convergence imply that $\supp\:\h{h}\subset \Gamma_0^+$ and $h(0)=a_0$. Of course any normally convergent Hahn
series with admissible support gives rise to a Hahn holomorphic function.

A direct consequence of the support of Hahn holomorphic functions being admissible is
\begin{lemma}\label{lem2.6}
Let
\[
h(z)=\sum_{\gamma\in \Gamma} a_{\gamma} e_\gamma(z), \quad z \in Y \cap D_{2 r}.
\]
be Hahn holomorphic with $\mathfrak{m}=\min\supp(\h{h})$. Then
\[
e_{-\mathfrak{m}}(z) h(z)=\sum_{\gamma\ge \mathfrak{m}} a_{\gamma} e_{\gamma-\mathfrak{m}}(z)
\]
is Hahn holomorphic.
\end{lemma}
\begin{proof}
Let $\rho_1$ be the radius for $\{e_\gamma\}$ such that for all $\gamma\in\supp(\h{h})$
\[
\infn{e_\gamma}{\rho_1(r)}\le |e_\gamma(r)|.
\]
and similarly let $\rho_2$ the radius for $\{e_{\gamma-\mathfrak{m}}\}$. For $\rho(r)=\min\{\rho_1(r),\rho_2(r)\}$,
\begin{align*}
\infn{e_{\mathfrak{m}}}{\rho(r)} \sum_{\gamma\in \Gamma}\|a_{\gamma}\| \infn{e_{\gamma-\mathfrak{m}}}{\rho(r)}
\le |e_{\mathfrak{m}}(r)| \sum_{\gamma\in \Gamma}\|a_{\gamma}\| |e_{\gamma-\mathfrak{m}}(r)|
=\sum_{\gamma\in \Gamma}\|a_{\gamma}\| |e_{\gamma}(r)|<\infty
\end{align*}
Thus $\sum_{\gamma\in \Gamma}a_{\gamma} e_{\gamma-\mathfrak{m}}$ converges normally on $D_{\rho(r)}$.
\end{proof}

\begin{proposition}\label{prop1.3}
Let $f: Y \to \mathcal{R}$ be a Hahn holomorphic function represented by a Hahn series $\h{f}$ on $Y \cap D_\delta$.
Suppose the zeros of $f$ accumulate in $Y \cup \{0\}$. Then $f\equiv 0$ and $\h{f}=0$. In particular the Hahn series of a Hahn
holomorphic function is completely determined by the germ of the function at zero.
\end{proposition}
\begin{proof}
If the zero set of $f$ has accumulation points in $Y$ then the statement follows
from the fact that $f$ is holomorphic in this set. It remains to show that if $f \neq 0$
then  $0$ can not be an accumulation point of the zero set of $f$. Let $\h{f}$ be a Hahn series that represents the function on $Y \cap D_\epsilon$.
Let $f\neq 0$, hence $\h{f}\neq 0$. Let $\h{m}=\min\supp\:\h{f}$. If there is no other element in the support of $\h{f}$ then
$f(z)=a_{\h{m}} e_{\h{m}}(z)$ and the statement follows from the fact that $e_{\h{m}}$ has no zeros in $Y$. Otherwise, let
$\h{m}_1$ be the smallest element in $\supp\:\h{f}$ which is larger than $\h{m}$. Then
\[
f(z)=\sum_\alpha a_\alpha e_\alpha(z)
=e_{\h{m}}(z)\Big(a_{\h{m}}+ e_{\h{m}_1-\h{m}}(z)\sum_{\alpha\ge \h{m}_1} a_\alpha e_{\alpha-\h{m}_1}(z)\Big)
=e_{\h{m}}(z) (a_{\h{m}}+ h(z))
\]
with a Hahn holomorphic function $h(z)$ such that $h(0)=0$.
Since $h$ is continuous and $e_{\h{m}}(z)\neq 0$ this shows $f(z)\neq 0$ in a  neighborhood of $0$.
\end{proof}

In the following suppose $Y, \Gamma$ and the family of functions $(e_\gamma)_{\gamma \in \Gamma}$
is fixed and satisfies (E1).

We want to show that the space of Hahn holomorphic functions at $0$ with values in a Banach algebra $\mathcal{R}$  is a ring. To that end we need

\begin{lemma}\label{lem-admis}
Let $A_1, A_2 \subset \Gamma^+$ be admissible sets. Then the sets $A_1\cup A_2$, $A_1+A_2$ and $n\cdot A_1:=A_1+\ldots+A_1$ ($n$ times), $\bigcup_{n=0}^\infty n\cdot A_1$  are admissible.
\end{lemma}
\begin{proof}
First we show that $A_1\cup A_2$, $A_1+A_2$ and $n\cdot A_1$ are weakly monotonous.
Let $\rho_i,\;i=1,2$ be the radius for $A_i$ and $\rho(x)=\min\{\rho_1(x),\rho_2(x)\}$.
Then $\rho$ is a radius for $A_1\cup A_2$ and as well for $A_1+A_2$, because
for $\alpha_i\in A_i,$ \begin{align*}
\infn{e_{\alpha_1+\alpha_2}}{\rho(r)}
&\le \infn{e_{\alpha_1}}{\rho(r)}\infn{e_{\alpha_2}}{\rho(r)}
\le \infn{e_{\alpha_1}}{\rho_1(r)}\infn{e_{\alpha_2}}{\rho_2(r)}\\
&\le |e_{\alpha_1}(r)||e_{\alpha_2}(r)|=|e_{\alpha_1+\alpha_2}(r)|
\end{align*}
The same argument shows that $\rho_1$ is a radius for $n\cdot A_1$.

Now let $B\subset A:=A_1+A_2$. Then $B=B_1+B_2$ for some
$B_i\subset A_i, i=1,2$ and $\min B=\min B_1+ \min B_2$. Let $\alpha\in A$ with $\alpha=\alpha_1+\alpha_2, \alpha_i\in A_i$.
Let $\rho_{i}(r)$ be the radius for $\{e_{\alpha_i-\min B_i}\}$ and
$\rho=\min\{\rho_1,\rho_2\}$. The estimate
\[
\infn{e_{\alpha-\min B}}{\rho(r)}= \infn{e_{\alpha_1-\min B_1+\alpha_2-\min B_2}}{\rho(r)}
\le \infn{e_{\alpha_1-\min B_1}}{\rho_{1}(r)}\infn{e_{\alpha_2-\min B_2}}{\rho_{2}(r)}
\]
shows that $A_1+A_2$ is admissible. The other statements are proven similarly. \end{proof}

Let $f(z)=\sum_\alpha a_\alpha e_\alpha$ and $g(z)= \sum_\beta b_\beta e_\beta$ be Hahn holomorphic functions on $Y_f$ and $Y_g$ respectively.
First it is easy to see that $f+g$ is Hahn holomorphic on $Y=Y_f\cap Y_g$.
Since $\h{f}$ and $\h{g}$ are Hahn-series with support contained in $\Gamma^+_0$, also $\supp(\h{f}\cdot\h{g})\subset \Gamma^+_0$ for the multiplication as defined in \eqref{hahn-prod}. From Lemma \ref{lem-admis} we obtain that the support of  $\h{f}\cdot\h{g}$ is admissible.
We claim that
$
h(z)=f(z)\cdot g(z)
$
is represented by the  product of Hahn-series $\h{h}= \h{f}\cdot\h{g}$ on $Y_f\cap Y_g$.
Because $f$ and $g$ are normally convergent,
\begin{multline*}
\sum_\gamma \big\|\!\!\sum_{\alpha+\beta=\gamma} a_\alpha b_\beta \big\| \infn{e_\gamma}{}
\le \sum_\gamma \Big(\sum_{\alpha+\beta=\gamma} \|a_\alpha\| \|b_\beta\| \Big) \infn{e_\gamma}{}\\
\le
\big(\sum_\alpha \|a_\alpha\| \infn{e_\alpha}{}\big) \big(\sum_\beta \|b_\beta\| \infn{e_\beta}{}\big)
\end{multline*}
so that the series $\h{f}\cdot\h{g}$ is normally convergent in $Y_f\cap Y_g$.
Thus the series $\h{f}\cdot\h{g}$
defines a Hahn holomorphic function on $Y$ with values in ${\mathcal R}$ which equals $h(z)$.

Altogether we have found

\begin{proposition}\label{ringhom} Let $\mathcal{R}$ be a Banach algebra.
The Hahn holomorphic functions with values in $\mathcal{R}$ on $Y$ form a ring under usual addition and multiplication,
and the map $\psi_{\mathcal R}:f\mapsto \h{f}$ is a ring isomorphism onto its image in $\mathcal{R}[[e_\gamma]]$.
\end{proposition}

\begin{corollary}\label{cor1.8}
The ring of Hahn holomorphic functions  on  $Y$ with values in an integral domain ${\mathcal R}$ is an integral domain.
\end{corollary}
\begin{proof}
By looking at the coefficient $c_{\gamma}$ with $\gamma=\min\supp\:\h{f}$ in \eqref{hahn-prod}, we observe that ${\mathcal R}[[e_\Gamma]]$ is an integral domain, if ${\mathcal R}$ is an integral domain. Because $\psi_{\mathcal R}$ is an isomorphism, the Hahn holomorphic functions must be an integral domain.
\end{proof}

\begin{theorem}\label{dasInverse}
Let $\mathcal{R}$ be a Banach algebra and suppose $f: Y_0 \to \mathcal{R}$ is Hahn holomorphic
and $f(z)$ is invertible for all  $z \in Y_0$.
Then $f(z)^{-1}$ is also Hahn holomorphic on $Y_0$.
\end{theorem}

\begin{proof}
Since $1/f$ is holomorphic in $Y$ we only have to show
that there is a Hahn series for $f(z)^{-1}$ that converges normally on some $Y_0 \cap D_\eps$.  Since $f(z)^{-1}= f(0)^{-1}\left( f(z) f(0)^{-1} \right)^{-1}$
we can assume without loss of generality that $f(0)=\Id$.
Thus we can write
$f(z) = \Id-h(z)$, where $\mathfrak{m}:=\min\supp\:(\h{h})>0$.
By assumption the series $\h{h} := \sum_{\alpha\in\Gamma}  a_\alpha  e_\alpha$ defining $h(z)$ converges normally on the  set $Y_0 \cap D_{\delta_0}$ for some $\delta_0>0$. The function $ \tilde{h}$ defined by
\[
\tilde{h}(t)=\sum_{\alpha\in\Gamma} \| a_\alpha \| \infn{e_\alpha}{Y_0,t}
\le \infn{e_{\mathfrak{m}}}{Y_0,t} \sum_{\alpha\ge \mathfrak{m}} \| a_\alpha\| \infn{e_{\alpha- \mathfrak{m}}}{Y_0,t}
\]
converges to $0$ for $t\to 0$  due to (E1) and Lemma \ref{lem2.6}.
Therefore we can  choose $\delta>0$ so small that $ \tilde h:=\tilde{h}(\delta) < 1/2$.
Because $ |h(z)| \leq \tilde h$ for $z\in Y_0 \cap D_\delta$, the geometric series
\[
f(z)^{-1} = \sum_{n=0}^\infty h(z)^n
\]
then converges normally on $Y_0 \cap D_\delta$ . But we also know that $\h{f}$ is invertible:
\[
\h{f}^{-1}= \sum_{n=0}^\infty \h{h}^n=:\sum_{\alpha\in\mathcal{S}}  b_\alpha   e_\alpha,
\quad\text{with}\quad \supp(\h{f}^{-1})\subset
\mathcal{S}:=\bigcup_{n\ge 0}\supp(\h{h}^{n}). 
\]
From Lemma \ref{lem-admis} we obtain that $\mathcal{S}$ is admissible.
It remains to show that $\sum_{\alpha\in\mathcal{S}}  b_\alpha   e_\alpha(z)$ is normally convergent on $Y_0 \cap D_\delta$ and represents $f(z)^{-1}$.
Note that if $\sum_{n=0}^N \h{h}^n= \sum_{\alpha\in\mathcal{S}} c_\alpha(N) e_\alpha$ then
\[
\sum_{\alpha\in\mathcal{S}} \| c_\alpha(N) \| \infn{e_\alpha}{} \leq \sum_{n=0}^N {\tilde{h}}^n\quad\text{in}\quad Y_0 \cap D_\delta
\]
as a simple consequence of the triangle inequality.
For every fixed finite set $A \subset \mathcal{S}$ there exists an $N_{\!A}>0$ such that for all $N \geq N_{\!A}$
\[
\h{f}^{-1} - \sum_{n=0}^{N} \h{h}^n=\sum_{\alpha\in\mathcal{S}\setminus A}(b_\alpha-c_\alpha(N)) e_\alpha\label{diff01}
\]
has support away from $A$.
In particular $c_\alpha(N)=b_\alpha$ for $\alpha\in A$ and $N\ge N_{\!A}$. Therefore for $N>N_{\!A}$
\[  \sum_{\alpha \in A}\| b_\alpha \| \infn{e_\alpha}{}
\leq  \sum_{\alpha\in \mathcal{S}} \|c_\alpha(N)\| \infn{e_\alpha}{}  \leq  \sum_{n=0}^{N} \tilde h^n<\frac{1}{1-\tilde{h}},
\] and this proves convergence since this bound is independent of $A$. 
In particular $\sum_{\alpha\in\mathcal{S}}  b_\alpha   e_\alpha(z)$ converges absolutely in $\mathcal R$, hence it converges and the value does not depend on the order of summation. After reordering,
\[
\sum_{\alpha\in\mathcal{S}}  b_\alpha   e_\alpha(z)=\sum_{n=0}^\infty h(z)^n=f(z)^{-1}.
\]
\end{proof}

Because of Lemma \ref{lem2.6}, every complex valued Hahn holomorphic $f$ which is not identically $0$ can be inverted away from its zeros: Let $\mathfrak{m}:=\min\supp(\h{f})\ge 0$, then
\[
f^{-1}(z)=a_{\mathfrak{m}}^{-1}e_{-\mathfrak{m}}(z)\sum_{n=0}^\infty\big(1-a_{\mathfrak{m}}^{-1}e_{-\mathfrak{m}}(z)f(z)\big)^n
\]

\begin{theorem}\label{holomsatz}
Suppose that $f: Y_0 \to \Complex$ is a Hahn holomorphic function with Hahn series
$\h{f}$. Suppose that $U$ is an open neighbourhood of $f(0)$ and $h: U \to \Complex$
is holomorphic. Then $h \circ f$ is Hahn holomorphic on its domain.
\end{theorem}
\begin{proof}
Since holomorphicity away from zero is obvious it is enough to show that $h \circ f$ has a normally convergent expansion into a Hahn series.
Replacing $f(z)$ by $f(z)-f(0)$ and $h(z)$ by $h(z-f(0))$ we can assume without loss of generality that
$f(0)=0$ and thus $\supp(\h{f}) \subset \Gamma^+$.
Since $h$ is holomorphic near $f(0)$ it has a uniformly and absolutely convergent expansion
$$
h(z) = \sum_{k=0}^\infty a_k (z-f(0))^k.
$$
Thus,
$$
h \circ f(z) =  \sum_{k=0}^\infty a_k (f(z))^k.
$$
Note that $\sum_{k=0}^\infty a_k \:\h{f}^k$ is a Hahn series. A similar argument
as in the proof of Theorem \ref{dasInverse} shows that this Hahn series is normally convergent
and represents $h \circ f(z)$.
\end{proof}

\section{Hahn meromorphic functions} \label{section:mero}

\begin{definition}
A meromorphic function $h: Y \to \Complex$ is called \emph{Hahn meromorphic}
if $h$ is represented by a Hahn series $\h{h}$ in $Y\cap D_\eps$ for some $\eps>0$
and there exist Hahn holomorphic functions $f$, $g\not\equiv 0$ on  $Y_0\cap D_\eps$ such that
$\h{h}\cdot \h{g}=\h{f}$.
\end{definition}

In this sense a Hahn meromorphic function can be written as a quotient $h=f/g$ of Hahn holomorphic functions in a neighborhood of $0$.

\begin{rem}
Since $\Complex$-valued Hahn holomorphic functions form an integral domain, the Hahn meromorphic functions
form a field.
More generally let $\mathcal{R}$ be a (commutative) integral domain.
From Corollary \ref{cor1.8} we know that Hahn holomorphic functions with coefficients in $\mathcal{R}$ are a commutative integral domain, so that their quotient field is defined. Furthermore, the map $f\mapsto \h{f}$ induces an injective morphism from the quotient field of Hahn holomorphic functions to the quotient field $\mathcal{R}((e_\Gamma))$ of Hahn series $\mathcal{R}[[e_\Gamma]]$. 
Note that $\mathcal{R}((e_\Gamma))=\mathcal{R}[[e_\Gamma]]$, if $\mathcal{R}$ is a field.
\end{rem}

An important difference with usual meromorphic functions is that Hahn meromorphic functions
may have infinitely many negative exponents. For example the function
\[
f(x) = \sum_{n=1}^\infty \frac{1}{n^2} z^{1-1/n}
\]
is Hahn holomorphic and therefore \[
\sum_{n=1}^\infty \frac{1}{n^2} z^{-1/n-1}=\frac{f(z)}{z^2}
\]
is Hahn-meromorphic.

It follows from our analysis for Hahn holomorphic functions that every $\Complex$-valued Hahn meromorphic function $h$ can be written as
\[
h(z)=e_{\min\supp{\:\h{h}}}(z) f(z),
\]
where $f$ is Hahn holomorphic. Moreover, if $h \neq 0$ then $f(0)\neq 0$.
In particular this implies that Hahn meromorphic functions which are bounded
on $(0,\delta)$ are Hahn holomorphic in some neighborhood of $0$.

We can also define Hahn meromorphic functions with values in a Banach algebra:
\begin{definition}
Let $\mathcal{R}$ be a Banach algebra.
A function $h:Y\to \mathcal{R}$ is called \emph{Hahn meromorphic} if it is meromorphic on $Y$ and there exists
a $\delta>0$ and a non-zero Hahn holomorphic function $f$ on $Y_0 \cap D_\delta$ such that
$f(z) h(z)$ is a Hahn holomorphic function on $Y_0\cap D_\delta$ with values in $\mathcal{R}$.
\end{definition}

\begin{rem}
Let $R>0$ and $\sigma>0$. If there exists one non-zero Hahn holomorphic function on $Y \cap D_R^{[\sigma]}$ that vanishes with
positive order at $0$,
then one can use 
the Weierstrass product theorem together with Theorem \ref{holomsatz} to show that
the set of complex valued Hahn meromorphic functions on $Y \cap D_R^{[\sigma]}$ can be identified with the quotient field
of the division ring of Hahn holomorphic functions on $Y \cap D_R^{[\sigma]}$.
\end{rem}

\section{A Hahn holomorphic  Fredholm theorem} \label{section:fred}
Let $\mathcal{H}$ be a complex Hilbert space and denote by $\mathcal{K}(\mathcal{H})$ the space of
compact operators on $\mathcal{H}$. 
\begin{theorem}\label{fredholm-thm}
Suppose $Y_0 \subset \ZZ$ is connected and
let $f: Y \to \mathcal{K}(\mathcal{H})$ be either Hahn holomorphic, or Hahn meromorphic such that
all coefficients of $e_\gamma$ with $\gamma<0$ and all Laurent coefficients in the principal part away from the point $z=0$
have range in a common finite dimensional subspace $\mathcal{H}_0 \subset \mathcal{H}$.
\vspace{-\parskip}

Then either
$(\Id-f(z))\in\mathcal{B}(\mathcal{H})$ is invertible nowhere in $Y_0$ or its inverse
$
(\Id-f(z))^{-1}
$
exists everywhere except at a discrete set of points in $Y_0$ and defines a  Hahn meromorphic
function. Moreover, in the Hahn series of $(\Id-f(z))^{-1}$,  the coefficients   of $e_\gamma$ with $\gamma<0$ are finite rank operators, and the coefficients in the principal part 
of its Laurent expansion  away from $z=0$ are finite rank operators, too.
\end{theorem}

\begin{proof}
The proof generalizes that of Theorem VI.14 of  \cite{Reed-Simon}.
The assumptions imply that there exists a Hahn meromorphic function $B(z)$ with range in $\mathcal{H}_0$,
a finite rank operator $A$, and a $\delta>0$ such that $f(z) - A - B(z)$ is Hahn-holomorphic and
$\| f(z) - A - B(z) \| < 1$ for all $z \in U^{[\sigma]}:= D_\delta^{[\sigma]}\cap Y$. Thus
$(\Id - f(z)+A + B(z))^{-1}$ exists and is  Hahn holomorphic by Theorem \ref{dasInverse}.
Consequently $g(z)=\left( A + B(z) \right)(\Id - f(z)+A + B(z))^{-1}$ is a Hahn meromorphic function on $U^{[\sigma]}$ with values in the
Banach space $\mathcal{B}(\mathcal{H},V)$, where $V$ is the finite dimensional subspace of $\mathcal{H}$ spanned by $\mathcal{H}_0$ and $\mathrm{rg}(A)$. 
It is easy to see that
\begin{equation}\label{fredholm-g1}
(\Id-f(z))^{-1} = (\Id - f(z)+A +B(z))^{-1} (\Id-g(z))^{-1}
\end{equation}
where equality means here that the left hand side exists if and only of the right hand side exists.
Let now $P$ be the orthogonal projection onto $V$ and let $G(z)$ be the endomorphisms
of $V$ defined by restricting $g(z)$ to $V$, i.e.~$G(z)=g(z)\circ P$.
Invertibility of $(\Id-g(z))$ in $\mathcal{B}(\mathcal{H})$ is equivalent
to invertibility of 
\[
P(\Id-g(z))P:V \to V,
\]
and this is equivalent to $\det(\Id_{V}-G(z))\neq 0$.
Moreover, a straightforward computation shows
\begin{equation}\label{fredholm-g2}
(\Id-g(z))^{-1}=(P(\Id-g(z))P)^{-1} \:\big(P+g(z)(\Id-P)\big)+(\Id-P).
\end{equation}
Now note that $G(z)$ is a Hahn holomorphic family of endomorphisms of $V$. In particular $\det(\Id-G(z))$ is a Hahn meromorphic $\Complex$-valued function. As such, it is meromorphic in $ U^{[\sigma]}\setminus\{0\}$, and together with  Proposition \ref{prop1.3} this  shows that the set \[
S=\{z\in U^{[\sigma]} \mid \det(\Id-G(z))=0\}
\]
is either discrete in $U^{[\sigma]}$ or $S=U^{[\sigma]}$.
If $\det(\Id-G(z))\neq 0$, then after a choice of basis of $V$ the inverse $(\Id-G(z))^{-1}$ can be computed with Cramer's rule, showing that with respect to this basis
\[
\det(\Id-G(z)) (\Id-G(z))^{-1}\in \mathrm{Mat}\big(\dim V, \Complex[[e_\Gamma]]\big)
\]
is represented by a matrix with Hahn meromorphic entries.
After the identification
\[
\mathrm{Mat}\big(\dim V, \Complex[[e_\Gamma]]\big)=\mathrm{Mat}\big(\dim V, \Complex\big)[[e_\Gamma]]
\]
we see that the function $(\Id-G(z))^{-1}$ is Hahn meromorphic with coefficients in $\mathrm{End}(V)$ if there is only a single point in $U^{[\sigma]}$
for which it exists.  Consequently, due to \eqref{fredholm-g1} and \eqref{fredholm-g2}, $(\Id-f(z))^{-1}$
is Hahn meromorphic with all coefficients of $e_\gamma(z)$ with $\gamma<0$ being of finite rank, if there is only a single point in $U^{[\sigma]}$ for which  $(\Id-f(z))$ is invertible. 
So far we have proved the statement in $U^{[\sigma]}$. By the usual analytic Fredholm theorem, invertibility of  $(\Id-f(z))$ at a single point in
$Y$ implies that the inverse exists as a meromorphic function
on  $Y$. Conversely, we have seen that
invertibility of  $(\Id-f(z))$ at a single point in $U^{[\sigma]}$ implies that $(\Id-f(z))^{-1}$
exists as a Hahn meromorphic function on $U^{[\sigma]}$. By the usual meromorphic Fredholm theorem
it then exists as a Hahn meromorphic function on $Y$.
\end{proof}

\section{$z$-Hahn holomorphic functions} \label{section:z}
The prominent class of Hahn holomorphic functions is defined by convergent power series with non-integer powers.

Let $\Gamma\subset\Reell$ be a subgroup with order inherited from the standard ordering of $\Reell$.
As the group $G$ we will take the group generated by the set of functions \[
e_{\alpha}(z) := z^\alpha,\qquad \alpha \in \Gamma,\quad z\in D_r^{[\sigma]} \backslash \{0\}.
\]
In this definition we choose the principal branch of the logarithm with $|\Im\log z|<\pi$ for $z\in\Complex\setminus (-\infty,0]$ and as usual set
$\log(r e^{\rmi\varphi})=\log r+ \rmi\varphi,\;|\varphi|<\sigma$ and $z^\alpha:=e^{\alpha \log z}$.

A $z$-Hahn holomorphic function $f$ with values $\Complex$ then is a holomorphic function on $D_r^{[\sigma]} \backslash \{0\}$ such that the
generalized power series
\[
f(z)=\sum_\gamma a_\gamma z^\gamma,\qquad a_\gamma\in\Complex
\]
is normally convergent in $Y\cap D_\delta^{[\sigma]}$ for some $\delta>0$.

Note that every well-ordered subset of $W\subset \Gamma^+$ is admissible for $e$, because for every $\alpha\in W$,
\begin{equation}\label{power-monotonie}
|z^\alpha|= |z|^\alpha\le |z|^{\min W},\qquad z\in D_{1/2}^{[\sigma]}.
\end{equation}

\begin{example}
If $\Gamma=\mathbb{Z}$ and $e_k(z)=z^k$ then the set of Hahn series corresponds
to the formal power series and the set of $z$-Hahn holomorphic functions can be identified with the
set of functions that are holomorphic on the disc of radius $\delta>0$ centered at the origin. \end{example}

\begin{example} The
series
\[
z^\pi\sum_{k=0}^\infty \frac{z^{2k}}{(2k)!}
\]
converges normally on $D_r$ for any $r>0$ and defines a $z$-Hahn holomorphic function for $\Gamma=\pi\Ganz+2\Ganz$.
\end{example}

\begin{example}
Puiseux series and Levi-Civita series as defined in  e.g.~\cite{ribenboim} are special cases of Hahn series with certain $\Gamma\subset \Rational$.
In case they are normally convergent they define $z$-Hahn holomorphic functions.
\end{example}

In the following let $D_R=D_R^{[\infty]}\setminus\{0\}$ be the pointed disk of radius $R$ in the logarithmic covering of the complex plane. 
The next result is in analogy with complex analysis, where series expansions converge normally on the maximal disc embedded in the domain of holomorphicity:

\begin{theorem}\label{th5.4}
Let $\mathcal{R}$ be a Banach algebra and suppose $f$ is $z$-Hahn-holomorphic. Furthermore suppose that $f$ is bounded on $D_{\tilde R+\eps}$ for some $\eps,\tilde R>0$.

Let
\[
f(z)=\sum_{\alpha\in\supp f} a_\alpha z^\alpha
\]
be its expansion (which we do not assume to converge normally on $D_{\tilde R}$).

Then for all $R$ with $0<R<\tilde R$:
\[
\sum_{\alpha\in\supp f} \|a_\alpha\| R^\alpha \le \sup_{|z|= R} \|f(z)\| \sum_{\alpha\in\supp f} (R/\tilde R)^\alpha.
\]
In particular, if $\sum_{\alpha} (R/\tilde R)^\alpha<\infty$ then the Hahn-series converges normally on $D_{R}$.
\end{theorem}
\begin{proof}
As a Hahn holomorphic function, $f$ converges normally on $D_{2\delta}$ for some $\delta>0$ and is holomorphic in $D_{\tilde R}$.
Let $\Lambda_{R,L}$ be the averaging operator
\[
\Lambda_{R,L}(f)=\frac{1}{2\pi \rmi L}\int_{S_R^{(L)}}\frac{f(z)}{z}\,dz,
\]
where $S_R^{(L)}(t)=R e^{\rmi\pi t}, t\in (-L,L] $ is the $L$-fold cover of the circle with radius $R$. 
Certainly 
\[
\|\Lambda_{R,L}(f)\| \le \sup_{|z|=R} \|f(z)\|.
\]
Since $f$ is holomorphic for $0<|z|<R$, we have
\[
\frac{1}{2\pi \rmi L}\int_{S_R^{(L)}}\frac{f(z)}{z}\,dz= \frac{1}{2\pi \rmi L}\int_{S_\delta^{(L)}}\frac{f(z)}{z}\,dz+O(L^{-1}).
\]
This shows 
\[
\Lambda_R(f):=\lim_{L\to\infty}\Lambda_{R,L}(f)=\lim_{L\to\infty} \sum_{\alpha\in\supp f} \frac{a_\alpha}{2\pi \rmi L} \int_{S_\delta^{(L)}} z^{\alpha-1}\,dz = a_0.
\]

Suppose $(I_k)$ is
a family of finite subsets of $\supp f$ such that
\[
I_1\subset I_2\subset\ldots \quad\text{and}\quad \bigcup_k I_k=\supp f.
\] For $z\in D_{\tilde R}$ let
$
g_k(z)=\sum_{\alpha\in I_k} \lambda_\alpha z^{-\alpha} R^\alpha
$
where $\lambda_\alpha\in \mathcal{R}^*:=\mathcal{B}(\mathcal{R},\Reell)$ are chosen such that
\[
\|\lambda_\alpha\|=1,\qquad \lambda_\alpha(a_\alpha)=\|a_\alpha\|.
\]
Such $\lambda_\alpha$ exist by the Hahn-Banach theorem.

Then $g_k$ is holomorphic in $D_{\tilde R}$  and
$
\|g_k(z)\|\le\sum_{\alpha\in I_k}|z|^{-\alpha} R^\alpha.
$
Moreover
\[
\sce{g_k}{f}(z)=\sum_{\substack{\alpha\in I_k\\ \gamma=-\alpha+\beta}}\lambda_\alpha(a_\beta) R^\alpha z^\gamma
\]
and the constant term of this function is
\[
\sum_{\alpha\in I_k} \|a_\alpha\| R^\alpha = \Lambda_\delta(\sce{g_k}{f})= \Lambda_{\tilde R}(\sce{g_k}{f}).
\]
Therefore
\begin{align*}
\sum_{\alpha\in I_k} \|a_\alpha\| R^\alpha
&\le \sup_{|z|=R} |\sce{g_k}{f}(z)|
\le \sup_{|z|=R} \|g_k(z)\| \|f(z)\| \\
&\le \sup_{|z|=R} \|f(z)\| \sum_{\alpha\in \supp f} (R/\tilde R)^\alpha
\end{align*}
and the theorem follows by letting $k\to\infty$.
\end{proof}

\begin{theorem}
Let $R>0$ and assume $f_k\colon D_R\to V$ is a sequence of bounded $z-$Hahn-holomorphic functions that converge uniformly to a bounded function $f\colon D_R\to V$. Suppose that there exist constants $C>0, \hat\eps>0$ such that for each $k\in\Nat$
\[
\sum_{\alpha\in\supp f_k}{\hat\eps}^\alpha<C.
\]
Suppose furthermore that there exists $I\subset\Reell$ such that $\supp f_k \to I$ in the following sense:
For each compact subset $K\Subset \Reell$ there exists $N>0$ such that
\[
\supp f_k\cap K=I\cap K\quad\text{for all}\quad k\ge N.
\]

Then $f$ is Hahn-holomorphic on $D_R$ with $\supp f\subset I$.
\end{theorem}
\begin{proof}
First, $I$ is well-ordered because $\supp f_k\to I$. Let $f_k(z)=\sum_{\alpha\in \supp f_k}a_\alpha^{(k)} z^\alpha$ be the expansion of $f_k$.

Let $\eps>0$, then there exists $N_1>0$ such that $\|f_{\ell}(z)-f_k(z)\|<\eps$ for all $k,\ell>N_1$ and all $z\in D_R$. Given a finite subset $\tilde I\subset I$, we can choose $N>N_1$
such that $\tilde I \cap \supp f_k=\tilde I$ for all $k>N$.
Theorem \ref{th5.4} then shows
for all $k,\ell>N$ and $\tilde R<\hat\eps R$ that
\begin{align*}
\sum_{\alpha\in \tilde I} \|a_\alpha^{(\ell)}-a_\alpha^{(k)}\|\cdot {\tilde R}^\alpha 
&\le \sup_{|z|=\tilde R}\|f_\ell(z)- f_k(z)\|\cdot\hspace{-2em}\sum_{\alpha\in\supp f_k\cup \supp f_\ell}\hspace{-2em}{\hat\eps}^\alpha
< 2\, C\,\eps.
\end{align*}
It follows that $(a_\alpha^{(k)})_k$ is a Cauchy sequence for each $\alpha$.
Let $a_\alpha:=\lim\limits_{k\to\infty}a_\alpha^{(k)}$.
Given a finite subset $\tilde I\subset I$ and $\eps>0$ we can find $N$ such that $\|a_\alpha^{(k)}-a_\alpha\|<\eps$ for all $k>N, \alpha\in\tilde I$. Then for $|z|<\tilde R<\hat\eps$,
\[
\Big\| \sum_{\alpha\in\tilde I}a_\alpha^{(k)} z^\alpha - \sum_{\alpha\in\tilde I}a_\alpha z^\alpha \Big\|
<\sum_{\alpha\in \tilde I} \|a_\alpha^{(k)}-a_\alpha\| {\tilde R}^\alpha
<C\,\eps.
\]
This shows that $\sum_{\alpha\in I}a_\alpha z^\alpha $ is a Hahn series for $f$. By the uniform convergence of $(f_k)$, $f$ is analytic in $D_R^{[\sigma]}\setminus\{0\}$. Its Hahn series converges normally on $D_{\tilde R}$ because
\begin{align*}
\sum_{\alpha\in\tilde I}\|a_\alpha\| {\tilde R}^\alpha
&\le \sum_{\alpha\in\tilde I}\|a_\alpha^{(\ell)}\| {\tilde R}^\alpha
+\sum_{\alpha\in\tilde I}\|a_\alpha^{(k)}-a_\alpha\| {\tilde R}^\alpha
+\sum_{\alpha}\|a_\alpha^{(k)}-a_\alpha^{(\ell)}\| {\tilde R}^\alpha \\
&\le \sum_{\alpha\in\supp f_\ell}\|a_\alpha^{(\ell)}\| {\tilde R}^\alpha
+C\eps
+2 C \eps <\infty.
\end{align*}
for all finite $\tilde I\subset I$, $\ell$ sufficiently large, and $k\gg \ell$ depending on $\tilde I$.
\end{proof}

\section{$z \log z$-Hahn holomorphic functions} \label{section:zlogz}

In the following let $\Reell^2$ be equipped with the lexicographical order and
let $\Gamma \subset \Reell^2$ be a subgroup with order inherited from that of $\Reell^2$.
Let $Y=D^{[\sigma]}_{1/2}$ for fixed $\sigma>0$.
The group $G$ will be generated by
\[
e_{(\alpha,\beta)}(z) := z^\alpha (-\log z)^{-\beta},\qquad (\alpha,\beta)\in \Gamma,\quad |z|<1.
\]
With the inclusion $\Reell\times\{0\}\subset \Reell^2$
this comprises the power functions $z^\alpha$ from the previous section.
Note that
\[
\lim_{z\to 0}e_{(\alpha,\beta)}(z)=0\iff \alpha>0 \lor (\alpha=0 \land \beta>0)
\]
which is equivalent to $(\alpha,\beta)>(0,0)$ in the lexicographical ordering of $\Reell^2$. 
The monotonicity \eqref{power-monotonie} of power functions $z^\alpha$
has to be replaced by the following ``weak monotonicity'' property.
\begin{lemma}\label{lem6.1}
Let $\mathcal{S}\subset \Gamma^+=\{\gamma\in\Gamma\mid\gamma>0\}$ be a set such that there exists an $N\in\Nat_0$
with
\begin{equation}
-\beta\le N \alpha \quad\text{for all}\quad (\alpha,\beta) \in \mathcal{S}.\tag{$\ast$}
\end{equation}
Then
\begin{enumerate}[a)]
\item\label{teil1} There exists $r_N<1$ such that for $(\alpha,\beta)\in\mathcal S$ and $|\theta|<\sigma$ the function
\[
r\mapsto |r e^{\rmi\theta}|^\alpha |\log (r e^{\rmi\theta})|^{-\beta}
\]
is monotonously increasing on $[0,{{r}_N})$.

\item\label{teil2} Given $x$ with $0<x<{r}_N$, there exists $\rho_N(x)\le x$ such that for all $z$ with $0\le |z|\le \rho_N(x), |\arg z|<\sigma$ we have
\[
(\alpha,\beta)\in\mathcal S \quad\Longrightarrow\quad |e_{(\alpha,\beta)}(z)| \le
e_{(\alpha,\beta)}(x)
\] \end{enumerate}
\end{lemma}
\begin{proof}
The proof is elementary and will be omitted here.
\end{proof}
It is not difficult to see that if $\mathcal{S}$ satisfies $(\ast)$, then a similar inequality holds for the set
$
(\mathcal{S} - A)\cap \Gamma^+
$
where $A\subset\mathcal{S}$ and the constant $N$ depends on $A$.
Thus a set $\mathcal{S}$ with $(\ast)$ is admissible for $e$.

Now the assumptions from section \ref{abschnittHHol}  are all satisfied and we can consider Hahn holomorphic and meromorphic functions:
A \emph{$z\log z$-Hahn holomorphic function} with values in a Banach algebra $\mathcal{R}$ is defined by a normally convergent series
\[
f(z)=\sum_{(\alpha,\beta)\in \Gamma} a_{(\alpha,\beta)} z^\alpha(-\log z)^{-\beta},\qquad a_{(\alpha,\beta)}\in\mathcal{R},\qquad z \in D^{[\sigma]}_{1/2},
\]
such that $\supp(f)$ is contained in a set $\mathcal{S}\cup\{(0,0)\}$ with $\mathcal{S}$ as in Lemma \ref{lem6.1}.

Note that the property ($\ast$) is invariant under addition and multiplication of Hahn holomorphic functions, so that
$z\log z$-Hahn holomorphic functions indeed are a ring, and all results from section \ref{abschnittHHol} apply.

\begin{example}
The series
\[
\sum_{n=0}^\infty z^n(-\log z)^{n}=(1+z\log{z})^{-1}
\]
is a Hahn series in $\Gamma=\Ganz\times\Ganz$ with support
$
\{ (n,-n) \mid n\in\Nat_0\}.
$
It converges normally on the set $\{z \in \ZZ \mid | z \log z | <1/2 \}$
and therefore defines a $z\log z$-Hahn holomorphic function on  $D^{[\sigma]}_r$ for any $\sigma>0$ and sufficiently small $r=r(\sigma)$.
\end{example}

\begin{example}
The formal series \[
\sum_{n=0}^\infty \frac{1}{n!} {z} (-\log {z})^{n}
\]
is \emph{not} a Hahn series for $\Gamma=\Ganz\times\Ganz$, because the support \[
\{ (1,-n) \mid n\in\Nat_0\}
\]
is not a well-ordered subset of $\Gamma$.
\end{example}

\begin{example}
The logarithm $\log z=\frac{z\log z}{z}$ is Hahn meromorphic for $\Gamma\subset \Ganz\times\Ganz$.
\end{example}

\begin{example}
The series
\[
\sum_{n=1}^\infty\sum_{m=1}^\infty \frac{1}{m^2} z^n (-\log z)^{(2n-1+\frac{1}{m})}
\]
defines a $z\log z$-Hahn holomorphic function in a neighborhood $D^{[\sigma]}_\epsilon$ of $0$ for any $\sigma>0$ and for small enough $\epsilon=\epsilon(\sigma)$ with $\Gamma=\Ganz\times\Rational$. Its support is
\[
\{ (n,1-2n-1/m) \mid n,m\in\Nat \}.
\]
\end{example}

\section{Applications: Hahn meromorphic continuation of resolvent kernels} \label{section:beispiel}
\subsection{}\label{section:71}
Suppose that $\nu\ge 0$. Then the differential operator $B_\nu$ associated to the Bessel differential equation in its Liouville normal form,
\begin{equation}\label{besselop}
B_\nu:= - \frac{\partial^2}{\partial x^2} +\frac{\nu^2-\frac{1}{4}}{x^2}\Id,
\end{equation}
is a non-negative symmetric operator on  the space $C_c^\infty((0,\infty))$
equipped with the inner pro\-duct inherited from $L^2((0,\infty), dx)$. 
In the following we will  denote the Friedrichs extension of $B_\nu$ by the same symbol $B_\nu$.

The kernel $r_\lambda^{(\nu)}$ of the resolvent $(B_\nu-\lambda^2)^{-1}$ can  be constructed directly out of the fundamental system of the corresponding Sturm-Liouville
equation  and this results in (see e.g. \cite[pg. 371]{bruening-seeley})
\begin{equation}\label{defkernel}
r_\lambda^{(\nu)}(x,y)=\frac{\rmi\,\pi}{2} \sqrt{xy}\cdot J_\nu\big(\lambda\min(x,y)\big) H_\nu^{(1)}\big(\lambda \max(x,y)\big),\quad 0< x,y< \infty,
\end{equation}
where
 $H_\nu^{(1)}$ is the Hankel functions of order $\nu$ of the first kind and $J_\nu$ is the Bessel function.

The proof of the following Lemma uses the well-known expansion of Bessel and Hankel functions, and will be given at the end of this section.

\begin{lemma}\label{lemkernbd}
For every $\nu>0$ and $(x,y)\in (0,\infty)\times(0,\infty)$ the kernel $\lambda\mapsto r_\lambda^{(\nu)}(x,y)$ defines a $z\log z$-Hahn-holomorphic function.

\begin{enumerate}[a)] \item For  $\nu\in\Reell_+\setminus\Nat_0$, 
\[
r_\lambda^{(\nu)}(x,y)=\lambda^{2\nu} f_1^{(\nu)}(x,y)(\lambda)+ f_2^{(\nu)}(x,y)(\lambda),
\]
where $\lambda\mapsto f_j^{(\nu)}(x,y)(\lambda)$ are even and entire. In particular $r_\lambda^{(\nu)}(x,y)$ is $z$-Hahn-holomorphic with support contained in $2\Ganz+2\nu\Ganz$.

\noindent Let $a_{j;2k}^{(\nu)}(x,y)$ be the coefficient of $\lambda^{2k}$ in the Taylor series expansion of $f_j^{(\nu)}(x,y)$.

\noindent There is a constant $C_1$, such that for $0\le x\le y$
\begin{align*}
|a_{1;2k}^{(\nu)}(x,y)| 
&\le R^{-2k} C_1(\nu) (x y)^{\nu+1/2} e^{R(x+y)},&  R&>0.\\
\intertext{For  $c>0$ and every $r_0>0$ there is a constant $C_2$, such that for all $y\ge x\ge c$}
|a_{2;2k}^{(\nu)}(x,y)|
&\le R^{-2k}\frac{C_2(\nu,r_0)}{\sqrt{R}}   \sqrt{x}(x/y)^\nu  e^{R (x+y)}, & R&\ge r_0.
\end{align*}
\item The kernel $\lambda\mapsto r_\lambda^{(\nu)}(x,y)$  is
 a $z\log z$-Hahn-holomorphic function with support contained in $2\Ganz \times \Ganz$, if  $\nu=n\in \Nat$:
 \[
r_\lambda^{(n)}(x,y)=\log(\lambda)  g_1^{(n)}(x,y)(\lambda)+ g_2^{(n)}(x,y)(\lambda),
 \]
where $\lambda\mapsto g_j^{(\nu)}(x,y)(\lambda)$ are even and entire.
 
 \noindent The coefficients $b_{j;2k}^{(n)}(x,y)$ in its Hahn-series expansion
 can be estimated by
 \begin{align*}
  |b_{1;2k}^{(n)}(x,y)|&\le R^{-2k} \sqrt{xy}\frac{(R/2)^{2n}}{(n!)^2}e^{R(x+y)},& R&>0,\\
 |b_{2;2k}^{(n)}(x,y)|&\le R^{-2k} e^{R(x+y)} \Big(\hat{c}_1 x^{n+1}y^{n+1/2}\frac{(R/2)^{2n}}{ n!(n-1)!}+c_2\Big),& R&>0.
 \end{align*}

\end{enumerate}

\end{lemma}

\begin{rem}
For $\nu=0$, the expansion \eqref{rexpansion} below gives
\[
r^{(0)}_\lambda(x,y)= -\sqrt{xy}\,\log\frac{\lambda y}{2}+\underline{h}(x,y)(\lambda)
\]
with a Hahn-holomorphic function $\underline{h}(x,y)$. In particular, $\lambda\mapsto r^{(0)}_\lambda(x,y)$ is $z\log z$-Hahn-meromorphic.
\end{rem}

For $c>0$ let $\chi_c:[0,\infty)\to \Reell_+$ be a smooth cutoff-function with
\[
\chi_c(x)=\begin{cases}
0,& x\le c \\
1,&x\ge 2c
\end{cases} 
\]
Multiplication with this function  defines a bounded operator on $L^2((0,\infty))$. The ``restricted resolvent'' $\chi_c (B_\nu-\lambda^2)^{-1} \chi_c$ then is the bounded operator on $L^2((0,\infty))$ with integral kernel 
\[
\big(\chi_c\circ r_\lambda^{(\nu)}\big)(x,y):=  \chi_c(x)\cdot r_\lambda^{(\nu)}(x,y)\cdot \chi_c(y).
\]
\vspace{1ex}

\begin{proposition}\label{propforfunc}
Let $I=(0,\infty)$, $\nu>0$ and $c>0$. For any $\kappa>0$ and $\sigma>0$ the restricted resolvent
$\chi_c (B_\nu-\lambda^2)^{-1} \chi_c$ extends, as a function of $\lambda$, to a $z\log z$-Hahn holomorphic function on some neighborhood $D_r^{[\sigma]}$ of $0$ with
values in the compact operators 
\[
\mathcal{K}\big(L^2(I, e^{\kappa x} \,dx ) , \: L^2(I, e^{-\kappa x} \,dx )\big).
\]
\end{proposition}
\begin{proof} First let $\nu\notin \Nat_0$. In Lemma \ref{lemkernbd}a), let $r_0=R=\kappa/3$. 
Using
\begin{subequations}
\begin{align}
\int_c^\infty \int_c^\infty \min(x,y)\Big(\frac{\min(x,y)}{\max(x,y)}\Big)^{2\nu}e^{(2R-\kappa)(x+y)}\,dx\,dy \le C(\kappa)\label{a02}
\intertext{and}
\int_c^\infty\int_c^\infty (x y)^{2\nu+1}e^{(2R-\kappa)(x+y)}\,dx\,dy \le \Big(\frac{\Gamma(2+2\nu)}{(\kappa/3)^{2+2\nu}}\Big)^2\label{a01}
\end{align}
\end{subequations}
it is easy to see that the coefficients $a_{j;2k}^{(\nu)}(x,y)$ of the Hahn series expansion of $r^{(\nu)}$ satisfy
\[
|\chi_c\circ a_{j;2k}^{(\nu)}| \in L^2(I\times I, e^{-\kappa (x+y)}\,dx\otimes dy),\quad j=1,2.
\]
Therefore the kernels $\{\chi_c\circ a_{j;2k}^{(\nu)}\}_k$ define  Hilbert-Schmidt-Operators 
\[
A_{j;2k}^{(\nu)}: L^2(I, e^{\kappa x} \,dx ) \to  L^2(I, e^{-\kappa x} \,dx)=:\mathcal{H}_{\kappa}
\]
with norm bounded from above by  
\begin{equation}\label{constbound1}
\|A_{j;2k}^{(\nu)}\|\le \|\chi_c\circ a_{j;2k}^{(\nu)}\|_{\mathcal{H}_{\kappa}\times \mathcal{H}_{\kappa}} \le R^{-2 k} C(\nu,\kappa),
\end{equation}
where $C(\nu,\kappa)$ can be obtained from \eqref{firstconstant},\eqref{a02},\eqref{secondconstant},\eqref{a01}.
But then the series
\[
\lambda^{2\nu}\sum_{k=0}^\infty \|A_{1;2k}^{(\nu)}\| |\lambda|^{2k}+\sum_{k=0}^\infty \|A_{2;2k}^{(\nu)}\| |\lambda|^{2k}
\]
converges normally in some neighborhood $U\subset D_r^{[\sigma]}$ of $0$ and the kernel $r_\lambda^{(\nu)}$
defines a $z$-Hahn-holomorphic family of Hilbert-Schmidt operators in
\[
\mathcal{K}\big( L^2\big(I, e^{\kappa x} \,dx ),\: L^2(I, e^{-\kappa x} \,dx)\: \big).
\]
For integral $\nu=n\in\Nat$ we can argue similarly, using Lemma \ref{lemkernbd}b).
\end{proof}

Finally we come to the

\begin{proof}[\textbf{Proof of Lemma \ref{lemkernbd}}]
First let $\nu\notin \Nat_0$.

Recall
\begin{align*}
J_\nu(z)&=\left(\tfrac{z}{2}\right)^{\nu}h_\nu(z),\quad h_\nu(z)=\sum_{k=0}^\infty a_k^{(\nu)} z^{2k},\quad 
\quad\text{with}\quad      {a}_{k}^{(\nu)}=\frac{(-1)^k}{4^k\,k!\Gamma(k+\nu+1)}
 \\
 \intertext{The function $h_\nu$ is entire.}
H_\nu^{(1)}(z)&=\frac{\rmi}{\sin\nu\pi} \Big(J_\nu(z)e^{-\rmi\nu\pi}-J_{-\nu}(z)\Big),\quad H_n^{(1)}(z)=\lim_{\nu\to n}H_\nu^{(1)}(z),\quad n\in\Ganz
\end{align*}

\textit{a)} Let $x\le y$. Then
\[
-\tfrac{2\rmi}{\pi}r_\lambda^{(\nu)}(x,y)=\sqrt{x y}J_\nu(\lambda x)H_\nu^{(1)}(\lambda y)=\lambda^{2\nu} f_1^{(\nu)}(x,y)(\lambda)+ f_2^{(\nu)}(x,y)(\lambda),
\]
with even, analytic functions in $\lambda$:
\begin{align*}
f_1^{(\nu)}(x,y)(\lambda)&=\frac{\rmi  e^{-\rmi\nu\pi}}{ 4^\nu \sin\nu\pi} (x y)^{\nu+1/2} h_\nu(\lambda x)h_\nu(\lambda y) \\
f_2^{(\nu)}(x,y)(\lambda)&=\frac{-\rmi}{ \sin\nu\pi} \sqrt{x y}\left(\frac{x}{y}\right)^\nu  h_\nu(\lambda x)h_{-\nu}(\lambda y) 
\end{align*}
Due to Cauchy's integral formula
\[
|a_{j,2k}^{(\nu)}(x,y)|\le R^{-2k} \sup_{|\lambda|=R} |f_j^{(\nu)}(x,y)|,\qquad j=1,2.
\]
We know from e.g. \cite[eq. (10.14.4)]{dlmf-bessel} that for $\nu\ge  0$ 
\begin{equation}
|h_\nu(z)|\le \frac{e^{|\Im z|}}{\Gamma(\nu+1)}. \label{boundh1}
\end{equation}
Using $J_{\nu}=\frac{1}{2}(H_\nu^{(1)}+ H_\nu^{(2)})$, $|H_{-\nu}^{(*)}|=|H_{\nu}^{(*)}|$, and that $h_{-\nu}$ is a holomorphic and even function: 
\[
(R/2)^\nu \sup_{|z|=R}|h_{-\nu}(z)|=\sup_{\substack{|z|=R\\ \Re(z)>0}}|J_{-\nu}(z)|
\le \sup_{\substack{|z|=R\\ \Re(z)>0}} \frac{1}{2}\Big( |H_\nu^{(1)}(z)|+ |H_\nu^{(2)}(z)| \Big)
\]
But from \cite[eq. (10.17.13)]{dlmf-bessel} for $-\frac{\pi}{2}<\arg(z)<\frac{\pi}{2}$,
\[
|H_\nu^{(1;2)}(z)|\le \sqrt{\frac{2}{\pi|z|}}e^{\mp \Im z} \Big(1+\tau_\nu(|z|)e^{\tau_\nu(|z|)}\Big),\qquad
\tau_\nu(s):=\tfrac{\pi}{2}|\nu^2-\tfrac{1}{4}|\cdot s^{-1}
\]
Thus there exists a constant $C_1>0$ with
\[
\sup_{\substack{|\lambda|=R\\ \Re(\lambda)>0}} |H_\nu^{(1;2)}(\lambda y)|\le C_1 \frac{e^{R y}}{\sqrt{R y}}(1+\tau(R) e^{\tau(R)}) ,
\qquad  y\ge c,\quad \tau(R):=\frac{\pi}{2}\frac{|\nu^2-1/4|}{R c}
\]
This shows that for every $r_0>0$ there is a constant $C$, such that for every $R\ge r_0$
\begin{subequations}
\begin{align}
|a_{2;2k}^{(\nu)}(x,y)|
&\le R^{-2k-\nu}\frac{C_\nu}{\sqrt{R}}   \sqrt{x}(x/y)^\nu  e^{R (x+y)},\qquad C_\nu:=\frac{C\cdot 2^\nu(1+\tau(r_0) e^{\tau(r_0)})}{|\sin(\nu\pi)|\Gamma(\nu+1)} \label{firstconstant}\\
\intertext{Also from \eqref{boundh1}:}
|a_{1;2k}^{(\nu)}(x,y)| &\le R^{-2k} \frac{ (x y)^{\nu+1/2} e^{R(x+y)}}{ 4^\nu |\sin\nu\pi|\,\Gamma(\nu+1)^2}.\label{secondconstant}
\end{align}
\end{subequations}

\vspace{2ex}
\textit{b)} Let $\nu=n\in\Nat$, and $x\le y$. Then  from $H_n^{(1)}=J_n+ \rmi Y_n$ and \cite[eq. (10.8.1)]{dlmf-bessel},
\begin{multline}\label{rexpansion}
\tfrac{-2\rmi}{\pi\sqrt{x y}}r_\lambda^{(n)}(x,y)=\frac{2\rmi}{\pi}(\log\lambda+\log\tfrac{y}{2})\cdot J_n(\lambda x)J_n(\lambda y)+J_n(\lambda x)J_n(\lambda y) \\
-\frac{\rmi}{\pi}h_n(\lambda x)\sum_{k=0}^{n-1}\frac{(n-k-1)!}{k!}(\lambda y/2)^{2k} \\
-J_n(\lambda x)\Big(\frac{\lambda y}{2}\Big)^n\frac{1}{\pi}\sum_{k=0}^\infty \frac{\psi(k+1)+\psi(n+k+1)}{k!(n+k)!}(-1)^k(\lambda y/2)^{2k}
\end{multline}
with $\psi(x)=\Gamma'(x)/\Gamma(x)$.
The only logarithmic terms in the Hahn-series expansion of $r^{(n)}(x,y)$ are $e_{(2k,-1)}(\lambda), k\ge n$.
Because of \eqref{boundh1}, the coefficient of $e_{(2k,-1)}$ is bounded by
\[
R^{-2k} \sqrt{xy}\frac{(R/2)^{2n}}{(n!)^2}e^{R(x+y)},\quad R>0.
\]
From Stirling's inequalities for $\Gamma$ we obtain for $0\le k\to\infty$
\begin{equation}\label{stir}
1\ge \sqrt{k+1}\frac{(2k)!}{4^k (k!)^2}\sim \frac{1}{\sqrt{\pi}},
\end{equation}
hence
\[
\frac{(n-k-1)!(2k)!}{4^k k! n!} =\frac{1}{(n-k)\binom{n}{k}}\cdot \frac{(2k)!}{4^k (k!)^2} \le \frac{1}{n}, \qquad 0\le k<n.
\]
Because $|z|^{2k}\le (2k)! e^{|z|}$ and  \eqref{boundh1}, the norm of the sum in the second line of \eqref{rexpansion} can be bounded by
$
\frac{e^{|\Im\lambda| x}}{\pi} e^{|\lambda y|}.
$

For the last line in \eqref{rexpansion} we first note that
the polygamma function is monotonously increasing and $\psi(k)\lesssim \log(k), k>0$ and estimate as above
\begin{gather*}
\sum_{k=0}^\infty \left|\frac{\psi(k+1)+\psi(n+k+1)}{k!(n+k)! 4^k}\right| |\lambda y|^{2k}
\le e^{|\lambda y|} \sum_{k=0}^\infty \frac{1}{\sqrt{k+1}}\frac{2\log(n+k+1)}{(k+n)\cdot(k+n-1)\cdots(k+1)} \\
\le 2 e^{|\lambda y|} \sum_{k=0}^\infty \frac{1}{\sqrt{k+1}(k+n)^{2/3}}\frac{\log(n+k+1)}{(k+n)^{1/3}}\frac{1}{(n-1)!}
\le \frac{c_1\, e^{|\lambda y|}}{(n-1)!}
\end{gather*}
where the constant $c_1$ can be obtained from $\zeta(7/6)$ and $x^{-1/3}\log(x+1)<\frac{3}{2}$ for $x>0$.

Altogether this shows that the coefficient of $\lambda^{2k}$ in $r_\lambda^{(\nu)}(x,y)$ is bounded by
\[
C R^{-2k}\sqrt{xy} e^{R(x+y)} \Big((\log(y/2)+1)\frac{(xy)^{n} (R/2)^{2n}}{(n!)^2}+\frac{1}{\pi}+c_1 (xy)^n \frac{(R/2)^{2n}}{n!(n-1)!}\Big),\quad R>0.
\]
and for $y\ge x\ge c$ this is smaller than 
\[
 R^{-2k} e^{R(x+y)} \Big(\hat{c}_1 x^{n+1}y^{n+1/2}\frac{(R/2)^{2n}}{ n!(n-1)!}+c_2\Big),\qquad R>0.
\]
\end{proof}

\vspace{2ex}

\subsection{The resolvent of the Laplace-Operator on cones.}~

Let $Z=(0,\infty)\times M$ be equipped with the cone metric $g^Z=dx^2+x^2 g^M$, where $(M,g^M)$ is a compact $n$-dimensional Riemannian manifold (without boundary); we will call $Z$ a \emph{cone}.
We consider the Friedrichs  extension $\Delta$ of the Laplace operator on compactly supported functions $C_0^\infty(Z)$ to $L^2(Z,g^Z)$. 
Under the isometry
\[
 \Psi: L^2(Z,dx^2+g^M) \to L^2(Z,g^Z), \qquad f(x,p)\mapsto x^{-n/2}f(x,p)
\]
this Laplacian becomes
\[
\widetilde\Delta:=\Psi^{-1}\circ\Delta\circ\Psi=-\frac{\partial^2}{\partial x^2}+\frac{1}{x^2}\left(\Delta_M+\tfrac{n}{2}\left(\tfrac{n}{2}-1\right)\right).
\]
Let $\{\mu_k\}$ be the
eigenvalues of the Laplace operator $\Delta_M$ on $L^2(M)$ and define $\nu_k:=\nu(\mu_k)$ as the positive solution of $\nu_k^2-\frac{1}{4}=\frac{n}{2}\left(\frac{n}{2}-1\right)+\mu_k$.
Let $V$ be the set of these solutions and  let $\{\phi_\nu\}_{\nu\in V}$ be the corresponding orthonormal Hilbert space basis of $L^2(M)$ consisting of eigenfunctions of $\Delta_{M}$, such that $\Delta_{M}\phi_{\nu(\mu)}=\mu\phi_{\nu(\mu)}$.

For a smooth function $f(x,p)=\sum_{\nu\in V} f_\nu(x)\phi_\nu(p)\in L^2(Z)$ we obtain
\[
\Psi^{-1}(\Delta-\lambda^2)\Psi f(x,p)=\sum_{\nu\in V}\big((B_\nu-\lambda^2)f_\nu\big)(x) \phi_\nu(p),
\]
where $B_\nu$ is the Bessel operator defined in \eqref{besselop}.

Let $\lambda\in\Complex$ with $\Im \lambda>0$, in particular $\lambda^2$ lies in the resolvent set of $\Delta$. Then the integral kernel of the resolvent $(\widetilde\Delta-\lambda^2)^{-1}$ is given by
\begin{equation}\label{theresolventkernel}
K((x,p),(y,q),\lambda)=\sum_{\nu\in V} r^{(\nu)}(x,y)(\lambda) \phi_\nu(q)\otimes \phi_\nu(p),
\end{equation}
where $r^{(\nu)}$ is defined in \eqref{defkernel}.
Recall from Lemma \ref{lemkernbd} that $r^{(\nu)}$ is a $z\log z$-Hahn holomorphic function,
\[
r^{(\nu)}(x,y)(\lambda)=\sum_{\gamma\in \tilde S_\nu\subset\Reell^2}a^{(\nu)}_\gamma(x,y)e_\gamma(\lambda),\quad e_{(\alpha,\beta)}(\lambda):=\lambda^\alpha (-\log \lambda)^{-\beta},
\]
where $\tilde S_\nu$ is the Hahn series support of $r^{(\nu)}$.
In this expansion, logarithmic terms occur only for $\nu\in\Nat$.

Take $\mathcal{G}\subset \Reell^2$ to be the group generated by $\bigcup_\nu S_\nu$ for $S_\nu:= \bigcup_{x,y\in(0,\infty)}\tilde S_\nu(x,y)$. Then it is clear that the resolvent kernel is a Hahn-series with support in $\mathcal{G}\subset\Reell^2$:
\begin{align}
K((x,p),(y,q),\lambda)&:=\sum_{\nu\in V} r^{(\nu)}(x,y)(\lambda) \,\phi_\nu(q)\otimes \phi_\nu(p)\nonumber\\
&=\sum_{\nu\in V} \sum_{\gamma\in S_\nu}a_\gamma^{(\nu)}(x,y)e_\gamma(\lambda) \,\phi_\nu(q)\otimes \phi_\nu(p)\nonumber\\
&=\sum_{\gamma\in\mathcal{G}} e_\gamma(\lambda) \big(\hspace{-1em}\sum_{\nu\in V: \gamma\in S_\nu}a_\gamma^{(\nu)}(x,y)\,  \phi_\nu(q)\otimes \phi_\nu(p) \big)\nonumber\\
&=:\sum_{\gamma\in\mathcal{G}}  \tilde{a}_\gamma((x,p),(y,q))\:e_\gamma(\lambda) \label{kernelK}
\end{align}
where we have set $\tilde{a}_\gamma=0$, if $\gamma\notin \bigcup_{\nu\in V} S_\nu$.

To show normal convergence of the operator valued series defined by \eqref{kernelK}, we will make the additional assumption that each $\nu\in V$ either is an integer, or is not ``too close'' to an integer in the following sense:
\begin{definition}\label{suitablenu}
For $\kappa>0$ a family of orders $V\subset \Reell_{\ge 0}$  is called $\kappa$-\emph{suitable}, if the set
\begin{equation}\label{einebedingung}
\left\{  \big((2\kappa)^\nu \sin(\nu\pi)\Gamma(\nu+1)\big)^{-1} \;\Big|\; \nu\in V\setminus \Nat \right\}\quad\text{is bounded.}
\end{equation}
\end{definition}

\begin{example}\label{ex7.3}
For $M=S^n, n\ge 1$, the $n$-sphere equipped with the standard metric, it is well-known (see e.g.~\cite{shubin}, \S 22) that the eigenvalues of the Laplace operator on functions are
$
\mu_k=k(k+n-1), k\in\Nat_0 $
with multiplicity $m_k:=\binom{n+k}{n}-\binom{n+k-2}{n}$.
Then
$
\nu_k:=\nu(\mu_k):= \frac{n-1}{2}+k
$ is  (half-)integral for odd (even) $n$ and
$V=\big(\nu_0,\:\nu_1,\ldots,\nu_1,\:\nu_2,\ldots\big)$, where
each $\nu_j$ appears  $m_j$ times. For $n\ge 2$ all $\nu(\mu_k)$ are positive.
\end{example}

In section \ref{section:71} we defined a smooth cutoff function $\chi$,
which can be extended to a bounded operator $\chi$ on $L^2((0,\infty)\times M,dx)$ by setting
\[
\chi(f)(x,p)=\chi(x)f(x,p),\quad\text{for}\quad f\in C_0^\infty(Z), x\in (0,\infty), p\in M,
\]
and taking the closure.

The main result of this section is

\begin{theorem}\label{propext1.2}~
\nopagebreak\par
\vspace{-1ex}Let $c,\sigma,\kappa>0$ and assume that the family $V=\{\nu\}$ of orders is $\kappa$-suitable. Then the restricted resolvent 
$\chi_c (\widetilde\Delta-\lambda^2)^{-1} \chi_c$  extends, as a function of $\lambda$, to a $z\log z$-Hahn meromorphic function on some $D_r^{[\sigma]}$ with
values in
\begin{equation}\label{weightscpt}
\mathcal{K}\big(L^2(Z, e^{\kappa\,x^{2}} \,dx\otimes \vol_{M} ) , \: L^2(Z, e^{-\kappa\, x^{2}} \,dx\otimes \vol_{M} )\big),
\end{equation}
where the only term $\lambda^\alpha (- \log \lambda)^{-\beta}$ in its Hahn series expansion with $(\alpha,\beta)<0$ that possibly has a non-zero coefficient is the one with $(\alpha,\beta) = (0,-1)$,
and its coefficient has finite rank.
If $V$ does not contain $\nu=0$, then in a (possibly smaller) neighborhood of zero this function is Hahn holomorphic.
\end{theorem}
\begin{proof}
Let $A_\gamma$ be the operator on $L^2(Z)$ defined by the ``restricted kernel''
\[
(\chi_c\circ \tilde a_\gamma)((x,q),(y,q)):=\chi_c(x)\tilde a_\gamma((x,p),(y,q)) \chi_c(y)
\]
with $\tilde a$ from \eqref{kernelK}, so that $\sum_{\gamma\in\mathcal{G}} A_\gamma e_\gamma(\lambda)$ is the Hahn series
of the restricted resolvent.

As in \eqref{constbound1} in the proof of Proposition \ref{propforfunc}  we can estimate 
\[
 \|\chi_c\circ \tilde a_\gamma^{(\nu)}\|\le R^{-k(\gamma)} C(\nu,\kappa);
\]
now instead of \eqref{a01} we choose $R<c\, \kappa/4$ and use
\begin{equation}\label{aa01}
\int_c^\infty x^{2\nu+1}e^{2R x-\kappa x^{2}}\,dx \le 
\int_c^\infty x^{2\nu+1}e^{-\frac{\kappa}{2} x^{2}}\,dx \le 
\frac{\Gamma(\nu+1)}{2(\kappa/2)^{\nu+1}}.
\end{equation}
Because the family $V$ is $\kappa$-suitable, the constants $C(\nu,\kappa)$ are bounded in $\nu$.
Thus  the kernel
${A}_\gamma$ defines a Hilbert-Schmidt operator 
\[
A_\gamma: L^2(Z,e^{\kappa\,x^2}dx\otimes \vol_{M})\to L^2(Z,e^{-\kappa\,x^{2}} dx\otimes \vol_{M})
\]
between weighted $L^2$-spaces, with norm  bounded by 
\[
\|A_\gamma\|\le  \|\chi_c\circ \tilde a_\gamma\|\le \sup_{\nu:\gamma\in\supp r_\nu} \|\chi_c\circ a_\gamma^{(\nu)}\| \le C
\]
for all $\gamma\in \bigcup_\nu S_\nu$.

Therefore the Hahn series $\sum_{\gamma\in\mathcal{G}} A_\gamma e_\gamma(\lambda)$ is normally convergent in $D_\delta^{[\sigma]}$ for some $\delta>0$,
provided that
\[
\sum_{\gamma\in\mathcal{S}} \|e_\gamma\|_\delta<\infty.
\]
Due to Lemma \ref{lemkernbd} the support $\mathcal{S}$ is given by
\begin{equation}\label{thesupport}
\mathcal{S}=\bigcup_{\nu}\supp r_\nu=\mathcal{S}_{\mathsf{r}} \cup \mathcal{S}_{\mathsf{i}}\subset \Reell_+\times (-\Nat_0)
\end{equation}
where $\mathcal{S}_{\mathsf{i}}, \mathcal{S}_{\mathsf{r}}$ correspond to integer and, non-integer real coefficients $\nu$. Furthermore, elements in $\mathcal{S}_{\mathsf{i}}$ are of the form $(2sn+2\ell, -s)$ with $\ell\in\Nat_0, s\in\{0,1\}$, and those in  $\mathcal{S}_{\mathsf{i}}$ have the form $(2s\nu+2\ell,0), \ell\in\Nat_0,s\in\{0,1\}$ for $\nu$ non-integer.

For $0<|\lambda| < \delta<1$ and $\nu\in V\setminus \Nat_0$,
\begin{align*}
&\quad\sum_{\gamma\in\mathcal{S}_{\mathsf{r}}}|e_\gamma(\lambda)|
\le  \sum_{\ell\in\Nat_0}|\lambda^{2\ell}|+ \sum_\nu \sum_{\ell\in\Nat_0} |\lambda^{2\ell+2\nu}| 
\le \frac{1}{1-\delta^2} \Big(1+\sum_\nu|\lambda^{2\nu}|\Big)
\end{align*}

Now from Weyl's formula we obtain that there exists an $R>0$ such that
$
\sum_{\nu\in V} R^\nu<\infty.
$
This shows that for $|\lambda|<\min(\delta,\sqrt{R})$ the partial series $\sum_{\gamma\in\mathcal{S}_{\mathsf{r}}}|e_\gamma(\lambda)|$ converges absolutely.

Finally, for  $\nu=n\in\Nat$  we  use  $|\log\lambda||\lambda|^{2k}\le C_\sigma |\lambda|^{2k-1}$ to estimate $\sum_{\gamma\in\mathcal{S}_{\mathsf{i}}}|e_\gamma(\lambda)|$ by the geometric series.

Note that the only term that gives rise to a non-zero coefficient of $e_\gamma$ with $\gamma<0$ is the order $\nu=0$.
\end{proof}

\begin{rem} From the proof of Theorem \ref{propext1.2} it is clear that a similar statement holds, if the weights in  \eqref{weightscpt} are replaced by
$e^{\pm\kappa x^{1+\eps}}$ for any $\eps>0$.
\end{rem}

For the Laplace operator on differential forms $L^2(S^n,\Lambda^* T^* S^n)$, the eigenvalues $\mu$ are integers (cf.~\cite{ik-tan}, Theorem 4.2), and the corresponding $\nu=\nu(\mu_k,p)$ are square roots of integers. In this case we have
%
%
%
\begin{lemma}\label{rootsareok}
Any family $V=(\sqrt{q_i})_i$ with $q_i\in \Nat_0$ is $\kappa$-suitable for every $\kappa>0$.
\end{lemma}
\begin{proof}
First one shows for $n\in\Nat_0$ and $q$ with $n^2<q<(n+1)^2$ that
\[
\min\{\sqrt{q}-n,(n+1)-\sqrt{q}\}>\frac{1}{2(n+1)}
\]
and then uses $|\sin x\pi|>2|x|$ for $0<|x|<\frac{1}{2}$ to prove for $\nu=\sqrt{q}$:
\begin{equation*}\label{bnd123}
\frac{1}{(\nu+1)|\sin\nu\pi|}<1,\qquad \frac{1}{\nu|\sin\nu\pi|}<\frac{3}{2}.
\end{equation*}
Together with Stirling's formula for the asymptotics of $\Gamma(\nu)$ this shows \eqref{einebedingung}.
\end{proof}

Therefore Theorem \ref{propext1.2} has a straightforward extension to differential forms.
A similar statement can be proven for the Laplacian acting on differential forms on $(0,\infty)\times P^n(\Complex)$, where $P^n(\Complex)$ is equipped with the Fubini-Study metric. The eigenvalues for the Laplace operator on sections of $\Lambda^p T^* P^n(\Complex)$  have been computed in Theorem 5.2 of \cite{ik-tan}, they are integers.

\vspace{4mm}

\subsection{The resolvent of the Laplace-Operator on compact perturbations of $\mathbb{R}^n$ or conic spaces}~\label{section:7.3}

Let $(Z = (0,\infty) \times M, g^Z)$ be a cone as defined in the previous section and let $X$ be a Riemannian manifold that is isometric to $Z$ away from a compact set.
This means for some $a>0$ we can identify $X$ with $X = X_a \cup_{M_a} Z_a$, where $Z_a = [a,\infty) \times M$,
$M_a = \{a\} \times M$, and $X_a$ is a compact Riemannian manifold with boundary $M_a$.

In this section we denote  by $\Delta_0$ the self-adjoint operator on the cone that is obtained from the Friedrichs' extension of the Laplace operator
on $C_0^\infty(Z)$.
Let $\Delta$ be the Laplace operator acting on compactly supported functions on $X$  and let $L$ be a formally self-adjoint first order differential operator that is compactly supported
in $X_a$ for some $a>0$. Then, of course $P := \Delta + L$ is of Laplace type and therefore essentially self-adjoint on
compactly supported smooth functions.
We will denote its self-adjoint extension by the same symbol $P$ whenever there is no
danger of confusion.
It follows from standard results in perturbation theory that the essential spectrum of $P$
equals the essential spectrum of the Laplace operator on the cone, namely $[0,\infty)$.
Moreover, it is well known that the distributional kernel of the resolvent $(P-\lambda^2)^{-1}$
has a meromorphic continuation across the spectrum away from the point $\lambda=0$. Now Theorem \ref{fredholm-thm} allows us to refine this statement and show that
the resolvent kernel is Hahn-meromorphic at $\lambda=0$ if this is true for the (restricted) kernel of $(\Delta_0-\lambda^2)^{-1}$.
The precise statement is formulated in the following theorem.
\begin{theorem}\label{hahnperturbed}
Let $a>0,\kappa>0$ and suppose that for some $\sigma>0$ the restricted resolvent
$\chi_a (\Delta_0-\lambda^2)^{-1}\chi_a $ extends, as a function of $\lambda$, to a $z \log{z}$-Hahn meromorphic function on  $D_r^{[\sigma]}$ with values in  
\[
  \mathcal{K}\big(L^2(Z, e^{\kappa x^2} \,dx\otimes \vol_{M} ) , \: L^2(Z, e^{-\kappa x^2} \,dx\otimes \vol_{M} )\big)
\]
for a group $\Gamma \subset \mathbb{Z} \times \mathbb{Z}$, such that the range of the coefficients of $e_\gamma$ with $\gamma<0$ of its Hahn series are finite rank operators
with range contained in a fixed finite dimensional subspace.
Let $\tilde \Gamma$ be the subgroup of $\mathbb{Z} \times \mathbb{Z}$
generated by $\Gamma$ and  $2\mathbb{Z} \times \{0\}$.
Then, $(P-\lambda^2)^{-1}$ has an extension, as a function of $\lambda$, to a $z \log{z}$ Hahn meromorphic function on  $D_r^{[\sigma]}$ for the group $\tilde \Gamma$
with values in 
\[
\mathcal{K}\big(L^2(X, w(x) \vol_{X} ) , \: L^2(X, w(x)^{-1} \vol_{X} )\big),
\]
where $w(x)$ is any positive function on $X$ such that $w(x) = e^{\kappa x^2}$ on $Z_a$.
Moreover, in the Hahn series expansion of this extension, the coefficients of $e_\gamma$ with $\gamma<0$  are finite rank operators.
\end{theorem}
\begin{proof}
The proof is identical to the standard proof that the meromorphic properties of the resolvent do not change
under compactly supported topological or metric perturbations. The only difference is that we apply our
Hahn-meromorphic Fredholm theorem. For the sake of completeness we give the full argument here.
By assumption we can choose $b>a>0$ such that the operators
$\Delta_0$ and $P$ agree on $C^\infty_0(Z_a)$.
Suppose $\psi_1, \psi_2,\phi_1,\phi_2$ are smooth functions on $X$ such that
$\mathrm{supp}\;\phi_1 \subset X_b$ and $\mathrm{supp}\;\psi_1 \subset X_a$
and such that
\begin{gather*} \psi_1 + \psi_2  =1,\quad
\psi_1 \phi_1 + \psi_2 \phi_2 =1,\quad
\mathrm{dist}(\mathrm{supp}\; d \phi_i,\mathrm{supp}\; \psi_i)>0.
\end{gather*}
Now denote by $P_{0}$ the self-adjoint operator obtained from $P$ by imposing Dirichlet boundary conditions at $M_b$. 
Since $P$ is an elliptic operator and the boundary conditions are elliptic, $P_0$ has compact resolvent
and therefore $Q_1(\Lambda):=(P_0-\lambda^2)^{-1}$ is a meromorphic function with values in $\mathcal{B}(L^2(X_b))$
and the residues of its poles are finite rank operators. Let us denote by $Q_2(\lambda)$ the Hahn-meromorphic extension of
$(\Delta_0-\lambda^2)^{-1}$ that exists by assumption.
Then,
\[
Q(\lambda):= \phi_1 Q_1(\lambda) \psi_1 + \phi_2 Q_2(\lambda) \psi_2
\]
is a Hahn meromorphic family with values in \[\mathcal{K}\big(L^2(X, w(x) \vol_{X} ) , \: L^2(X, w(x)^{-1} \vol_{X} )\big)\]
with respect to the group $\tilde \Gamma$ and the coefficients of $e_\gamma$ with $\gamma<0$ of its Hahn series are finite rank operators
with range contained in a fixed finite dimensional subspace.
By construction, for $\lambda\in D_r^{[\sigma]}, \Im\lambda>0$,
\[
Q(\lambda)(P-\lambda^2)= \Id+K(\lambda)
\]
with
\[
K(\lambda)=K_1(\lambda)+K_2(\lambda),\quad K_i(\lambda):=\phi_i Q_i(\lambda) (\Delta \psi_i - 2 \nabla_{\mathrm{grad} \psi_i}).
\]
Since the integral kernels of $Q_i$ are smooth off the diagonal,
the operator
$
K(\lambda)
$
is smoothing. Moreover, its integral kernel has compact support in the second variable.

Since $Q_1(\lambda)$ is meromorphic and $Q_2(\lambda)$ is Hahn-meromorphic, and by the remarks before
$K(\lambda)$ is a Hahn meromorphic family with values in 
$
\mathcal{K}\big( L^2(X, w(x)^{-1} \vol_{X} )\big)
$
for the group $\tilde \Gamma$ and the coefficients of $e_\gamma$ with $\gamma<0$ of its Hahn series are finite rank operators
with range contained in a fixed finite dimensional subspace.
Furthermore, for $\lambda= \rmi  r$ purely imaginary one derives
$\|K_i(\rmi r)\|\le \frac{c}{r}$ for $r>1$. Therefore, for a sufficiently large $r$ the operator $\Id+K(\rmi r)$
is invertible. By the meromorphic Fredholm theory and Theorem \ref{fredholm-thm}, $(\Id+K(\lambda))^{-1}$
is a  family of operators in $\mathcal{K}\big(L^2(X, w(x)^{-1} \vol_{X} )\big)$ which is meromorphic away from zero with finite rank 
negative Laurent coefficients at its non-zero poles and finite rank coefficients of $e_\gamma$ with $\gamma<0$. 
It is Hahn meromorphic at zero for the group $\tilde \Gamma$.
Hence, we have
\[
(\Id + K(\lambda))^{-1}Q(\lambda)(P-\lambda^2) = \Id,
\]
and $(\Id + K(\lambda))^{-1}Q(\lambda)$ extends the resolvent of $P$ to a Hahn meromorphic function with the desired properties as claimed.
\end{proof}

Combining Theorem \ref{propext1.2} and Theorem \ref{hahnperturbed} we obtain
\begin{corollary}
Let $M$ be a Riemannian manifold that is isometric to $\mathbb{R}^n \backslash B_R, n\ge 2$ outside a compact set
for some sufficiently large $R>0$.
Let $P$ be a compactly supported perturbation of the Laplace operator in the sense of the Theorem \ref{hahnperturbed} and let $w(x)$ be as there. Then the resolvent $\lambda\mapsto(P-\lambda^2)^{-1}$
as a map 
\[
\{\Im\lambda>0\}\to \mathcal{K}\big(L^2(X, w(x) \vol_{X} ) , \: L^2(X, w(x)^{-1} \vol_{X} )\big),
\]
has a continuation to a function in $\lambda$ that is $z \log z$-Hahn meromorphic for the group $\Ganz\times \Ganz$.
\end{corollary}

When $n$ is odd, then from Example \ref{ex7.3} we conclude that $\Gamma=\mathbb{Z} \times \{0\}$. In this case  Theorem \ref{hahnperturbed} and its corollary are well known and follow from the usual meromorphic Fredholm theorem.
Similar convergent expansions in the case of two dimensional potential scattering with suitable decay at infinity 
were obtained in \cite{bgd}.
For example in \cite{bgd}, Theorem 3.3, it was shown by more direct methods that the transition operator $T(k)$ in $L^2(\Reell^2)$ has a convergent expansion in powers of $k$ and $\log k$.

\begin{rem}
Let $X$ be a Riemannian manifold with an end isometric to
\[
(Z=[1,\infty)\times N, dx^2+ x^{-2a} g^N), a>0
\]
for some closed Riemannian manifold $(N,g^N)$. The spectral theory of the Laplace operator on $X$
 is examined in detail in \cite{hrs}. There the authors show that the spectral decomposition of the Laplace operator on differential forms on $Z$ can also be described with the Weber transform. The same  arguments as in section \ref{section:71}, together with the proof of Theorem \ref{hahnperturbed}
then implies that the resolvent of the Laplace operator on $X$ is $z\log z$-Hahn meromorphic, provided that the eigenvalues of the Laplace operator on $N$ lead to suitable $\nu$.
\end{rem}

\begin{rem}
 Our method may also be applied to non-compactly supported perturbations of the Laplace operator on $\mathbb{R}^n$, 
 such as for example potential perturbations that have a suitable decay rate at infinity.
 This is in line with the well known result in the odd dimensional case, that uniform exponential decay of the potential at infinity guarantees the existence of an analytic continuation of the resolvent into a neighborhood of the spectrum.
\end{rem}

\end{document}